\documentclass{siamart0516}

\usepackage{lipsum}
\usepackage{amsfonts}
\usepackage{graphicx}
\usepackage{epstopdf}
\usepackage{algorithmic}
\usepackage{enumerate} 
\usepackage{mathtools} 
\usepackage{tkz-graph}
\usepackage{tikz}
\usetikzlibrary{arrows}
\usetikzlibrary{arrows.meta}

\usepackage{bbold}
\ifpdf
  \DeclareGraphicsExtensions{.eps,.pdf,.png,.jpg}
\else
  \DeclareGraphicsExtensions{.eps}
\fi

\crefname{lemma}{Lemma}{Lemmata}                      
\crefname{corollary}{Corollary}{Corollaries}          
\crefname{theorem}{Theorem}{Theorems}                 
\Crefname{lemma}{L}{Ls}                               
\Crefname{corollary}{C}{Cs}                           
\Crefname{theorem}{T}{Ts}                             

\newcommand{\TheTitle}{Existence of Positive Steady States for Weakly Reversible Mass-Action Systems} 
\newcommand{\TheAuthors}{B. Boros}

\headers{Positive Steady States of Mass-Action Systems}{\TheAuthors}

\title{{\TheTitle}\thanks{\funding{This work was supported by the Austrian Science Fund (FWF), project P28406.
}}}

\author{
Bal\'azs Boros\thanks{Faculty of Mathematics, University of Vienna, Oskar-Morgenstern-Platz 1, 1090 Wien, Austria
(\email{balazs.boros@univie.ac.at}, \url{https://web.cs.elte.hu/\string~bboros/}).}
}

\usepackage{amsopn}
\DeclareMathOperator{\ran}{ran}

\ifpdf
\hypersetup{
  pdftitle={\TheTitle},
  pdfauthor={\TheAuthors}
}
\fi

\begin{document}

\maketitle

\begin{abstract}
We prove the following. For each weakly reversible mass-action system, there exists a positive steady state in each positive stoichiometric class.
\end{abstract}



\section{Introduction}
\label{sec:intro}

Mass-action dynamical systems are probably the most common mathematical models in biochemistry, cell biology, and population dynamics. Reversible and weakly reversible mass-action systems are the most studied classes of such systems, and are ubiquitous in mathematical biology. Moreover, they represent a large class of polynomial dynamical systems that are very important both theoretically and from the point of view of applications.

The four authors Jian Deng, Martin Feinberg, Christopher Jones, and Adrian Nachman posted the manuscript \cite{deng:feinberg:jones:nachman:2011} on \url{arXiv.org} in 2011, claiming that there exists a positive steady state in each positive stoichiometric class for every weakly reversible mass-action system.
As mentioned in \cite[Remark 3.4]{deng:feinberg:jones:nachman:2011}, a result corresponding to \cref{lemma:Gzz_neg} of the present paper was established by Adrian Nachman in the early 1980s, and he obtained the existence of a positive steady state in each positive stoichiometric class for every weakly reversible mass-action system with a \emph{single} linkage class. However, his proof has never been made publicly available. Based on Adrian Nachman's work, roughly $20$ years later, Jian Deng attacked the problem of extending the result to the general case, i.e., for \emph{multiple} linkage classes. This resulted in the \url{arXiv.org} posting \cite{deng:feinberg:jones:nachman:2011} in 2011. The main feature of their approach is that the original problem, which is in the space of \emph{species}, is translated to the existence of a zero of a map in the space of \emph{complexes}. Then they use Brouwer's Fixed Point Theorem, directly to a ball in the single linkage class case, and for technical reasons to a modification of a ball in the two linkage classes case. They claim that similar arguments lead to resolve the case of three or more linkage classes.

The ideas of both Adrian Nachman and Jian Deng are important building blocks that are used and developed further in the present work. On the one hand, there was a need for a clarified presentation, on the other hand, the case of three or more linkage classes had to be worked out. This motivated the present work.

In \cite{deng:feinberg:jones:nachman:2011}, Deng et al.\ first concentrate on proving the existence of a positive steady state at all. Only at the very end, using the accumulated material, they address existence within positive stoichiometric classes. In the present paper, we right from the beginning consider one fixed positive stoichiometric class and show the existence of a positive steady state there. As a result, we quickly arrive to a complete proof in the single linkage class case, and the line of the presentation becomes straightforward.

We note that the existence of a positive steady state was already proven in some special cases, see the Deficiency-Zero and Deficiency-One Theorems in \cite{feinberg:1995a}, the reversible case in \cite{orlov:rozonoer:1984b}, the reversible case in two dimension in \cite{simon:1995}, and the deficiency-one case in \cite{boros:2013a}.

Finally, the authors of \cite{deng:feinberg:jones:nachman:2011} claim not only the existence, but also the finiteness of positive steady states in each positive stoichiometric class, assuming weak reversibility. However, their argument is insufficient. Moreover, with Gheorghe Craciun and Polly Yu, we have constructed counterexamples to this claim \cite{boros:craciun:yu:2018}.

The rest of this paper is organised as follows. After introducing some notations and the necessary notions from CRNT in \Cref{sec:notations,sec:mass_action_systems}, respectively, we state the main result of this paper in \Cref{sec:main_result}. We perform some preliminary steps of the proof of the main result in \Cref{sec:prelim_towards_main_result}, followed by the proof in the single linkage class case in \Cref{sec:proof_ell_1}. Before we turn to the proof in the multiple linkage classes case in \Cref{sec:proof_ell_general}, we meditate about it in \Cref{sec:meditation} and provide the proof for two linkage classes in \Cref{sec:proof_ell_2}. Finally, in \Cref{sec:app_birch,sec:proof_yi_exp_minusyiminus1,sec:acyclic_digraph}, we provide some details about a tool that is used in \Cref{sec:prelim_towards_main_result}, prove one of the lemmata of \Cref{sec:proof_ell_1}, and display the dependence of the numbered statements appearing in this paper via an acyclic digraph, respectively.

\section{Notations}
\label{sec:notations}

We use standard notations.

For two vector spaces $U$ and $V$, the notation $U \leq V$ expresses that $U$ is a subspace of $V$. For a subspace $U \leq \mathbb{R}^m$, the map $\Pi_U : \mathbb{R}^m \to \mathbb{R}^m$ is the orthogonal projection to $U$. For a subspace $U \leq \mathbb{R}^m$, the symbol $U^\perp$ denotes the orthogonal complement of $U$.

For a linear map (or its matrix) $A$, we denote by $\ker A$, $\ran A$, $\operatorname{rank}A$, and $A^\top$ its kernel, range, rank, and transpose, respectively.

For a vector $z \in \mathbb{R}^m$, we denote by $|z|$ and $\max(z)$ its Euclidean norm and the value of the maximal entry of $z$, respectively. For two vectors $z_1, z_2 \in \mathbb{R}^m$, the expression $\langle z_1, z_2 \rangle$ denotes their Euclidean scalar product.

For a finite set $Q$, we denote by $|Q|$ the number of its elements.

The element of $\mathbb{R}^m$ with all its coordinates being $1$ is denoted by $\mathbb{1}_m$.

For a set $\Omega \subseteq \mathbb{R}^m$, we denote by $\partial \Omega$ the boundary of $\Omega$.

The symbol $\mathbb{R}_+$ denotes the set of positive real numbers, i.e., $\mathbb{R}_+ = \{x \in \mathbb{R} ~|~ x > 0\}$.

For a vector $x \in \mathbb{R}^n_+$ and a matrix $Y \in \mathbb{R}^{n \times m}$, the vector $x^Y \in \mathbb{R}_+^m$ is defined by $(x^Y)_j = \prod_{i=1}^n x_i^{Y_{ij}}$ for $j \in \{1,\ldots,m\}$.

\section{Mass-action systems}
\label{sec:mass_action_systems}

We give a very brief introduction to the basic notions of CRNT. For more details, the reader is advised to consult e.g. \cite{feinberg:1979}, \cite{feinberg:1987}, and \cite{gunawardena:2003}.

A \emph{reaction network} is a triple $(\mathcal{X},\mathcal{C},\mathcal{R})$, where $\mathcal{X}$, $\mathcal{C}$, and $\mathcal{R}$ are the set of \emph{species}, \emph{complexes}, and \emph{reactions}, respectively. Throughout the paper, we use $n = |\mathcal{X}|$ and $m = |\mathcal{C}|$. The complexes are formal linear combinations of the species, the coefficients are stored in the matrix $Y \in \mathbb{R}^{n \times m}$. The $i$th complex is then $Y_{1i}\mathsf{X}_1 +\cdots + Y_{ni}\mathsf{X}_n$, where $\mathsf{X}_1, \ldots,\mathsf{X}_n$ denote the species. The set $\mathcal{R}$ consists of ordered pairs of complexes, the first and the second element of the pair are called \emph{reactant} complex and \emph{product} complex, respectively. The reactant and the product complex of a reaction are distinct.

The weak components of the digraph $(\mathcal{C},\mathcal{R})$ are called \emph{linkage classes}. The number of linkage classes is denoted by $\ell$. The reaction network is said to be \emph{weakly reversible} if all the weak components of the digraph $(\mathcal{C},\mathcal{R})$ are strongly connected, i.e., for every pair $(i,j)$ of complexes, the existence of a directed path from $i$ to $j$ implies the existence of a directed path from $j$ to $i$.

The \emph{incidence matrix} of the digraph $(\mathcal{C},\mathcal{R})$ is denoted by $I$, while its range is by $\mathcal{I}$. (Each column $v$ of $I$ corresponds to a reaction, has exactly two nonzero entries, $v_i = -1$ (respectively, $v_i = 1$) if the $i$th complex is the reactant (respectively, product) complex of the reaction in question.) Elementary considerations show that
\begin{align*}
\mathcal{I} = \left\{z \in \mathbb{R}^m ~\Big|~ \sum_{i \in \mathcal{C}(j)}z_i=0 \text{ for all } j \in \{1,\ldots,\ell\}\right\},
\end{align*}
where $\mathcal{C}(j)$ denotes the set of those complexes that belong to the $j$th linkage class. In \Cref{sec:meditation,sec:proof_ell_2,sec:proof_ell_general}, we will use the notation $m_j = |\mathcal{C}(j)|$.

Denoting by $x(\tau) \in \mathbb{R}^n_+$ the \emph{concentration vector} of the species at time $\tau$, assuming \emph{mass-action kinetics}, the time evolution of the species concentration vector is described by the autonomous ordinary differential equation (ODE)
\begin{align} \label{eq:ODE}
\dot{x}(\tau) = Y A_\kappa x(\tau)^Y \text{ with state space } \mathbb{R}^n_+,
\end{align}
where $\kappa : \mathcal{R} \to \mathbb{R}_+$ and the matrix $A_\kappa \in \mathbb{R}^{m \times m}$ is the \emph{Laplacian} of the labelled digraph $(\mathcal{C},\mathcal{R},\kappa)$. Namely,
\begin{align*}
A_\kappa =
\begin{bmatrix}
\kappa_{11} & \cdots & \kappa_{m1} \\
\vdots      & \ddots & \vdots      \\
\kappa_{1m} & \cdots & \kappa_{mm}
\end{bmatrix}-
\begin{bmatrix}
\sum_{i=1}^m\kappa_{1i} &        & 0                      \\
                        & \ddots &                        \\
0                       &        & \sum_{i=1}^m\kappa_{mi}
\end{bmatrix},
\end{align*}
where we implicitly set $\kappa_{ij}=0$ for $(i,j) \notin \mathcal{R}$. (In other situations, one might define the Laplacian as $A_\kappa^\top$, $-A_\kappa$, or $-A_\kappa^\top$, but the natural definition of the Laplacian in the field of CRNT is the one we gave.) The reason we defined the ODE in the positive orthant (not in the nonnegative orthant) is that we allow negative entries in $Y$. The quadraple $(\mathcal{X},\mathcal{C},\mathcal{R},\kappa)$ is called a \emph{mass-action system}.

The main object we are interested in in this paper is the set of \emph{positive steady states} $E_+$, defined by
\begin{align*}
E_+ = \{x \in \mathbb{R}^n_+~|~YA_\kappa x^Y=0\}.
\end{align*}

The matrix $S=YI$ is called the \emph{stoichiometric matrix}. In general, $\ran A_\kappa \leq \mathcal{I}$. (There are several ways to show that weak reversibility implies $\ran A_\kappa = \mathcal{I}$, see e.g. \cite[Appendix]{feinberg:horn:1977}, \cite[Corollary 4.6]{feinberg:1979}, \cite[Lemma V.2]{sontag:2001}, and \cite[Corollary 2.8]{boros:2013c}. In particular, assuming weak reversibility, $\ran A_\kappa$ is independent of $\kappa$.) Thus, the translations of $\ran S$ are forward invariant under the ODE \eqref{eq:ODE}. We call the sets $(p+\ran S) \cap \mathbb{R}^n_+$ for $p \in \mathbb{R}^n_+$ \emph{positive stoichiometric classes}. The relevant object to study is not $E_+$, but $E_+ \cap \mathcal{P}$, where $\mathcal{P}$ is a fixed positive stoichiometric class.

\section{Main result}
\label{sec:main_result}

The main result of this paper is the following theorem.

\begin{theorem}\label{thm:existence_in_each_class}
Let $(\mathcal{X}, \mathcal{C}, \mathcal{R}, \kappa)$ be a weakly reversible mass-action system and let $\mathcal{P}$ be a positive stoichiometric class. Then $E_+\cap\mathcal{P}\neq\emptyset$.
\end{theorem}

The way we prove \cref{thm:existence_in_each_class} is the following. In \Cref{sec:prelim_towards_main_result}, we arrive (via a series of lemmata) to \cref{thm:Epl_cap_P_hatGK_cap_Kperp}. This latter theorem provides an equivalent formulation to $E_+\cap\mathcal{P}\neq\emptyset$ (for an arbitrary mass-action system) in terms of an intersection problem in $\mathbb{R}^m$ (note that the original problem is in $\mathbb{R}^n_+$). Finally, \cref{thm:intersection_for_all_subspaces_and_all_F} states that weak reversibility is sufficient to the solvability of this intersection problem. The proof of this latter theorem under the extra assumptions $\ell = 1$, $\ell=2$, and $\ell\geq1$ are carried out in \Cref{sec:proof_ell_1,sec:proof_ell_2,sec:proof_ell_general}, respectively. Both of \Cref{sec:proof_ell_2,sec:proof_ell_general} build heavily on \Cref{sec:proof_ell_1}, however, \Cref{sec:proof_ell_2,sec:proof_ell_general} are independent of each other. The reason we spell out the proof of the case $\ell = 2$ is that it is a nice warm up for the case of arbitrary $\ell$ (the latter is slightly more abstract than other parts of this paper).

The \emph{Permanence Conjecture} states that for each weakly reversible mass-action system and each positive stoichiometric class $\mathcal{P}$ there exists a compact subset $K$ of $\mathcal{P}$ such that $K$ is forward invariant and is a global attractor when the dynamics is restricted to $\mathcal{P}$, see \cite{craciun:nazarov:pantea:2013} and \cite{gopalkrishnan:miller:shiu:2014}. Once the permanence conjecture is proved, \cref{thm:existence_in_each_class} follows immediately.

\section{Preliminary steps towards proving the main result}
\label{sec:prelim_towards_main_result}

In this section, we start analysing the question of the non-emptiness of the set $E_+ \cap \mathcal{P}$.

The set $E_+\cap\mathcal{P}$ lies in $\mathbb{R}^n_+$. The following lemma translates the question of its non-emptiness to an intersection problem in $\mathbb{R}^m$ of two manifolds, one linear and one nonlinear. The following definition will be used throughout this paper. Let us define the function $G : \mathbb{R}^m \to \mathbb{R}^m$ by
\begin{align*}
G(z) = A_\kappa e^z \text{ for } z \in \mathbb{R}^m.
\end{align*}

\begin{lemma} \label{lemma:GYTlogP_ranYTperp}
Let $(\mathcal{X}, \mathcal{C}, \mathcal{R}, \kappa)$ be a mass-action system and let $\mathcal{P}$ be a positive stoichiometric class. Then
\begin{align*}
E_+\cap\mathcal{P}\neq\emptyset \text{ if and only if } G(Y^\top\log\mathcal{P})\cap (\ran Y^\top)^\perp\neq\emptyset.
\end{align*}
\end{lemma}
\begin{proof}
By definition, $x \in E_+\cap\mathcal{P}$ if and only if $x \in \mathcal{P}$ and $A_\kappa x^Y \in \ker Y$. The observations $x^Y = e^{Y^\top\log x}$ and $\ker Y = (\ran Y^\top)^\perp$ conclude the proof.
\end{proof}

As \cref{lemma:GYTlogP_ranYTperp} suggests, we investigate the set $Y^\top\log\mathcal{P}$. Let us define the subspace $K \leq \mathbb{R}^m$ by $K = \Pi_\mathcal{I}(\ran Y^\top)$. This notation will be used only in \cref{lemma:use_birch_in_proof,lemma:hatGK_cap_Kperp}, \cref{thm:Epl_cap_P_hatGK_cap_Kperp}, and \cref{sec:app_birch}.

\begin{lemma} \label{lemma:use_birch_in_proof}
Let $(\mathcal{X},\mathcal{C},\mathcal{R})$ be a reaction network and let $\mathcal{P}$ be a positive stoichiometric class. Then there exists a unique function $F : K \to \mathcal{I}^\perp$ such that
\begin{align}\label{eq:monomial_parametrization}
Y^\top\log\mathcal{P} = \{z+F(z) ~|~ z \in K\}.
\end{align}
Moreover, $F$ is continuous.
\end{lemma}
\begin{proof}
By \cref{lemma:psi} in \cref{sec:app_birch}, the map $\Psi:\mathcal{P}\to K$, defined by $\Psi(x) = \Pi_\mathcal{I} Y^\top \log(x)$ is a bijection between $\mathcal{P}$ and $K$. Clearly, it is even a homeomorphism.

Since $\Psi$ is surjective, for all $z \in K$ there exists a $z'\in\mathcal{I}^\perp$ such that $z+z' \in Y^\top\log\mathcal{P}$. Since $\Psi$ is injective, such a $z'$ is unique. Thus, there exists a unique function $F : K \to \mathcal{I}^\perp$ such that \eqref{eq:monomial_parametrization} holds. Finally, since $F : K \to \mathcal{I}^\perp$ is actually defined by $F(z) = \Pi_{\mathcal{I}^\perp}Y^\top \log(\Psi^{-1}(z))$, it is continuous.
\end{proof}

Based on \cref{lemma:use_birch_in_proof}, we give another form of the set $G(Y^\top\log\mathcal{P}) \cap (\ran Y^\top)^\perp$.

\begin{lemma} \label{lemma:hatGK_cap_Kperp}
Let $(\mathcal{X},\mathcal{C},\mathcal{R},\kappa)$ be a mass-action system and let $\mathcal{P}$ be a positive stoichiometric class. Further, let $F : K \to \mathcal{I}^\perp$ be as in \cref{lemma:use_birch_in_proof}. Then
\begin{align*}
G(Y^\top\log\mathcal{P}) \cap (\ran Y^\top)^\perp = \widehat{G}(K) \cap K^\perp,
\end{align*}
where $\widehat{G}:K\to\mathbb{R}^m$ is defined by $\widehat{G} = G\circ(\operatorname{Id}+F)$.
\end{lemma}
\begin{proof}
Note that for any two subspaces $U$ and $V$ and for any element $u \in U$, we have $u \in V^\perp$ if and only if $u \in (\Pi_UV)^\perp$. Thus, with $U = \mathcal{I}$ and $V = \ran Y^\top$, we have
\begin{align*}
G(Y^\top\log\mathcal{P}) \cap (\ran Y^\top)^\perp = G(Y^\top\log\mathcal{P}) \cap K^\perp.
\end{align*}
By \cref{lemma:use_birch_in_proof}, the latter intersection equals to $\widehat{G}(K) \cap K^\perp$.
\end{proof}

\cref{thm:Epl_cap_P_hatGK_cap_Kperp} below is an immediate consequence of the \cref{lemma:GYTlogP_ranYTperp,lemma:hatGK_cap_Kperp}. It provides an equivalent condition to the non-emptiness of $E_+ \cap \mathcal{P}$ in terms of an intersection problem in $\mathbb{R}^m$. The highly nontrivial \cref{thm:intersection_for_all_subspaces_and_all_F} below states that this intersection problem, under weak reversibility, is always solvable.

\begin{theorem} \label{thm:Epl_cap_P_hatGK_cap_Kperp}
Let $(\mathcal{X},\mathcal{C},\mathcal{R},\kappa)$ be a mass-action system and let $\mathcal{P}$ be a positive stoichiometric class. Further, let $F : K \to \mathcal{I}^\perp$ be as in \cref{lemma:use_birch_in_proof}. Then
\begin{align*}
E_+\cap\mathcal{P} \neq \emptyset \text{ if and only if } \widehat{G}(K) \cap K^\perp\neq \emptyset,
\end{align*}
where $\widehat{G}:K\to\mathbb{R}^m$ is defined by $\widehat{G} = G\circ(\operatorname{Id}+F)$.
\end{theorem}

\begin{theorem} \label{thm:intersection_for_all_subspaces_and_all_F}
Let $(\mathcal{X}, \mathcal{C}, \mathcal{R}, \kappa)$ be a weakly reversible mass-action system. Let $H$ be an arbitrary subspace of $\mathcal{I}$ and $F:H\to \mathcal{I}^\perp$ be an arbitrary continuous function. Then $\widehat{G}(H)\cap H^\perp\neq\emptyset$, where $\widehat{G}:H\to\mathbb{R}^m$ is defined by $\widehat{G} = G \circ (\operatorname{Id}+F)$.
\end{theorem}

We conclude this section by meditating on a possible approach one can try to prove the above theorem (and we will indeed follow this way in the upcoming sections). Fix $H \leq \mathcal{I}$. Clearly, $\widehat{G}(H)\cap H^\perp\neq\emptyset$ if and only if $0 \in \Pi_H(\widehat{G}(H))$. Thus, our goal is to show that the map $\Pi_H \circ \widehat{G} : H \to H$ attains $0 \in \mathbb{R}^m$. By Brouwer's Fixed Point Theorem, it suffices to show that there exists an $R > 0$ such that we have $\langle \Pi_H(\widehat{G}(z)), z \rangle < 0$ for all $z \in H$ with $|z| = R$. Since $\langle \Pi_H(\widehat{G}(z)), z \rangle = \langle \widehat{G}(z), z \rangle$ for all $z \in H$, we will investigate the sign of the scalar product $\langle \widehat{G}(z), z \rangle$ for $z \in H$. As it will turn out, the way we just sketched indeed works under the extra assumption $\ell = 1$. To prove \cref{thm:intersection_for_all_subspaces_and_all_F} for arbitrary $\ell$, we have to do a little surgery on the ball $\{z\in H ~|~ |z|=R\}$.

\section[alternative section title for TOC]{Proof of \cref{thm:intersection_for_all_subspaces_and_all_F} under $\ell=1$} \label{sec:proof_ell_1}

In this section, we prove \cref{thm:intersection_for_all_subspaces_and_all_F} under the extra assumption $\ell=1$.

We start by an elementary lemma. Its proof is deferred to \Cref{sec:proof_yi_exp_minusyiminus1}.

\begin{lemma} \label{lemma:yi_exp_minusyiminus1}
Let $M \geq 1$ and $y_0 = 0$. Then
\begin{align*}
\inf_{y_1,\ldots,y_{M-1} \geq 0}\left(\sum_{i=1}^M e^{-y_{i-1}}y_i\right) \to \infty \text{ as } y_M\to\infty.
\end{align*}
\end{lemma}

Taking also into account the discussion at the end of \Cref{sec:prelim_towards_main_result}, the following lemma (whose proof is based on \cref{lemma:yi_exp_minusyiminus1}) concludes the proof of \cref{thm:intersection_for_all_subspaces_and_all_F} under the extra assumptions $\ell=1$ and $F \equiv 0$.

\begin{lemma} \label{lemma:Gzz_neg}
Let $(\mathcal{X}, \mathcal{C}, \mathcal{R}, \kappa)$ be a weakly reversible mass-action system with $\ell=1$. Then the following two statements hold.
\begin{enumerate}[(i)]
\item \label{it:L0} There exists an $L>0$ such that
\begin{align*}
\langle G(z),z \rangle < 0 \text{ for all } z \in \mathcal{I} \text{ with } \max(z) \geq L.
\end{align*}
\item \label{it:R0} There exists an $R>0$ such that
\begin{align*}
\langle G(z),z \rangle < 0 \text{ for all } z \in \mathcal{I} \text{ with } |z| \geq R.
\end{align*}
\end{enumerate}
\end{lemma}
\begin{proof}
It is an easy exercise to show that for all $L > 0$ there exists an $R>0$ such that $z \in \mathcal{I}$ (i.e., $\sum_{i=1}^mz_i=0$) and $|z|\geq R$ together imply $\max(z) \geq L$. Thus, once we show \eqref{it:L0}, the statement \eqref{it:R0} follows immediately. The rest of this proof is devoted to show \eqref{it:L0}.

Since
\begin{align*}
\langle G(z),z \rangle &= z^\top A_\kappa e^z =\\
&= \sum_{(i,j)\in\mathcal{R}}e^{z_i}\kappa_{ij}(z_j-z_i)=\\
&=e^{\max(z)}\sum_{(i,j)\in\mathcal{R}} \kappa_{ij}e^{-(\max(z)-z_i)}[(z_j-\max(z))+(\max(z)-z_i)],
\end{align*}
for any $\mathcal{R}' \subseteq \mathcal{R}$ we have
\begin{align*}
-e^{-\max(z)}\langle G(z),z \rangle = A + B + C,
\end{align*}
where
\begin{align*}
A &= \sum_{(i,j)\in\mathcal{R}'} \kappa_{ij}e^{-(\max(z)-z_i)}(\max(z)-z_j), \\
B &= \sum_{(i,j)\in\mathcal{R}\setminus\mathcal{R}'} \kappa_{ij}e^{-(\max(z)-z_i)}(\max(z)-z_j), \\
C &= -\sum_{(i,j)\in\mathcal{R}} \kappa_{ij}e^{-(\max(z)-z_i)}(\max(z)-z_i).
\end{align*}
Our goal is to show that for each $z \in \mathcal{I}$ with $\max(z)$ being big enough, one can choose $\mathcal{R}'$ such that $A+B+C$ is positive. Clearly, $B \geq 0$, because each term in the sum is nonnegative. Also, $C \geq -|\mathcal{R}|\max(\kappa)$, because $e^{-x}x\leq 1$ for all $x\geq 0$. To estimate $A$ from below, we will use \cref{lemma:yi_exp_minusyiminus1}.

Let $y_0 = 0$ and for fixed $M \in \{1,2,\ldots,m-1\}$, let $L_M > 0$ be such that
\begin{align*}
\sum_{i=1}^M e^{-y_{i-1}}y_i \geq 2|\mathcal{R}|\frac{\max(\kappa)}{\min(\kappa)} \text{ for all }y_1,\ldots,y_{M-1}\geq 0 \text{ and for all } y_M\geq L_M.
\end{align*}
Finally, let $L = \max(L_1,\ldots,L_{m-1})$. We will show in the rest of this proof that $A+B+C$ is positive for all $z \in \mathcal{I}$ with $\max(z)\geq L$.

Fix $z \in \mathcal{I}$. Since $\sum_{i=1}^mz_i = 0$, there exist $k$ and $l$ such that $z_k = \max(z)$ and $z_l \leq 0$. The digraph $(\mathcal{C},\mathcal{R})$ is assumed to be strongly connected, therefore, there exists a directed path from $k$ to $l$. Let $\mathcal{R}'$ be the edge set of this directed path, denote by $M\geq1$ the length of this path, and let $\pi(0),\pi(1),\ldots,\pi(M)$ be the enumeration of the vertices visited while travelling from $k$ to $l$ (thus $\pi(0)=k$ and $\pi(M)=l$). Further, let $y_i = \max(z) - z_{\pi(i)}$ ($i = 0,1, \ldots,M$). With this, $y_0 = 0$, $y_1,\ldots,y_M\geq 0$, and
\begin{align*}
A = \sum_{i=1}^M \kappa_{ij}e^{-y_{i-1}}y_i \geq \min(\kappa)\sum_{i=1}^M e^{-y_{i-1}}y_i.
\end{align*}
Then, since $y_M = \max(z) - z_l \geq \max(z)$, assuming $\max(z) \geq L$, we have
\begin{align*}
A + B + C \geq 2|\mathcal{R}|\frac{\max(\kappa)}{\min(\kappa)}\min(\kappa)+ 0 - |\mathcal{R}|\max(\kappa) = |\mathcal{R}|\max(\kappa) > 0.
\end{align*}
This concludes the proof.
\end{proof}

We mentioned before \cref{lemma:Gzz_neg} that the lemma concludes the proof of \cref{thm:intersection_for_all_subspaces_and_all_F} under the extra assumption $\ell = 1$ and $F \equiv 0$. As it is explained in the rest of this paragraph, not only for identically zero functions $F$. If $F:H\to\mathcal{I}^\perp$ is an arbitrary function (and $\ell=1$ holds) then all the coordinates of $F$ are equal, i.e., there exists a function $F_1:H \to \mathbb{R}$ such that $F(z)=F_1(z)\mathbb{1}_m$ for all $z \in H$. Thus, $\widehat{G}(z) = e^{F_1(z)}G(z)$ for all $z \in H$. Therefore, the sign of the scalar product $\langle \widehat{G}(z),z\rangle$ is independent of $F$. This concludes the proof of \cref{thm:intersection_for_all_subspaces_and_all_F} under the extra assumption $\ell = 1$.

We conclude this section by two more results. \cref{lemma:Gzz_neg_many_kappas} below is a stronger version of \cref{lemma:Gzz_neg}. It is apparent from the proof of \cref{lemma:Gzz_neg} \eqref{it:L0} that \cref{lemma:Gzz_neg_many_kappas} \eqref{it:L0_many_kappas} also holds (the positive constant $L$ can be chosen such that it serves its purpose uniformly for many $\kappa$'s at the same time). \cref{lemma:Gzz_neg_many_kappas} \eqref{it:R0_many_kappas} follows from \cref{lemma:Gzz_neg_many_kappas} \eqref{it:L0_many_kappas} exactly the same way as \cref{lemma:Gzz_neg} \eqref{it:R0_many_kappas} followed from \cref{lemma:Gzz_neg} \eqref{it:L0_many_kappas}. As a consequence of \cref{lemma:Gzz_neg_many_kappas} \eqref{it:R0_many_kappas}, we will obtain \cref{cor:Gzz_neg_error_rho} below. This latter corollary will play a crucial role in the proof of \cref{thm:intersection_for_all_subspaces_and_all_F} under $\ell \geq 2$, see \cref{lemma:Omega_is_good_ell_2,lemma:Omega_is_good_ell_general} in \Cref{sec:proof_ell_2,sec:proof_ell_general}, respectively.

\begin{lemma} \label{lemma:Gzz_neg_many_kappas}
Let $(\mathcal{X}, \mathcal{C}, \mathcal{R})$ be a weakly reversible reaction network with $\ell = 1$ and let $D>0$. Then the following two statements hold.
\begin{enumerate}[(i)]
\item \label{it:L0_many_kappas} There exists an $L>0$ such that for all $\kappa:\mathcal{R} \to \mathbb{R}_+$ with $\frac{\max(\kappa)}{\min(\kappa)} \leq D$ we have
\begin{align*}
z^\top A_\kappa e^z < 0 \text{ for all } z \in \mathcal{I} \text{ with } \max(z) \geq L.
\end{align*}
\item \label{it:R0_many_kappas} There exists an $R>0$ such that for all $\kappa:\mathcal{R} \to \mathbb{R}_+$ with $\frac{\max(\kappa)}{\min(\kappa)} \leq D$ we have
\begin{align*}
z^\top A_\kappa e^z < 0 \text{ for all } z \in \mathcal{I} \text{ with } |z| \geq R.
\end{align*}
\end{enumerate}
\end{lemma}

\begin{corollary} \label{cor:Gzz_neg_error_rho}
Let $(\mathcal{X}, \mathcal{C}, \mathcal{R}, \kappa)$ be a weakly reversible mass-action system with $\ell=1$. Then for all $\varrho \geq 0$ there exists an $R_\varrho>0$ such that
\begin{align*}
z^\top A_\kappa e^{z+w} < 0 \text{ for all } w \in \mathbb{R}^m  \text{ with } |w| \leq \varrho \text{ and for all } z \in \mathcal{I} \text{ with } |z| \geq R_\varrho.
\end{align*}
\end{corollary}
\begin{proof}
Motivated by $A_\kappa e^{z+w} = A_\kappa \operatorname{diag}(e^w) e^z$, let us define $\kappa^{(w)}:\mathcal{R} \to \mathbb{R}_+$ by $\kappa^{(w)}_{ij} = \kappa_{ij}e^{w_i}$ for $(i,j) \in \mathcal{R}$ ($\operatorname{diag}(e^w)$ is the diagonal matrix with the entries of $e^w$ on its diagonal). Since $\langle G(z+w),z\rangle = z^\top A_{\kappa^{(w)}} e^z$, $\inf_{|w|\leq\varrho}\min(\kappa^{(w)}) > 0$, and $\sup_{|w|\leq\varrho}\max(\kappa^{(w)}) < \infty$, \cref{lemma:Gzz_neg_many_kappas} \eqref{it:R0_many_kappas} immediately gives the result.
\end{proof}

\section[alternative section title for TOC]{Meditation about proving \cref{thm:intersection_for_all_subspaces_and_all_F} under $\ell\geq2$}
\label{sec:meditation}

Throughout this section, let $(\mathcal{X}, \mathcal{C}, \mathcal{R}, \kappa)$ be a weakly reversible mass-action system, $H$ be an arbitrary subspace of $\mathcal{I}$, and $F:H\to \mathcal{I}^\perp$ be an arbitrary function. We want to show that $\widehat{G}(H)\cap H^\perp\neq\emptyset$, where $\widehat{G}:H\to\mathbb{R}^m$ is defined by $\widehat{G} = G \circ (\operatorname{Id}+F)$. Following the ideas outlined at the end of \Cref{sec:prelim_towards_main_result}, it would be ideal if we could prove the existence of an $R > 0$ such that we have $\langle \widehat{G}(z), z \rangle < 0$ for all $z \in H$ with $|z| = R$. As we have already seen in \Cref{sec:proof_ell_1}, this approach works for $\ell = 1$. Under the restriction $F \equiv 0$, one could prove it also for arbitrary $\ell$. However, as the example in the next paragraph demonstrates, in case both $\ell$ and $F$ can be arbitrary, the existence of such an $R$ is not guaranteed.

Let us consider a mass-action system, for which
\begin{align*}
A_\kappa =
\begin{bmatrix*}[r]
-1 &  1 &  0 &  0 \\
 1 & -1 &  0 &  0 \\
 0 &  0 & -2 &  1 \\
 0 &  0 &  2 & -1
\end{bmatrix*}.
\end{align*}
Further, let $H$ be the span of $u$ and $v$, where
\begin{align*}
u =
\begin{bmatrix*}[r]
-1 \\ 1 \\ 0 \\ 0
\end{bmatrix*}
\text{ and }
v =
\begin{bmatrix*}[r]
0 \\ 0 \\ -1 \\ 1
\end{bmatrix*}.
\end{align*}
Finally, let $F:H\to\mathcal{I}^\perp$ be a linear function with
\begin{align*}
Fu =
\begin{bmatrix*}[r]
-2 \\ -2 \\ 0 \\ 0
\end{bmatrix*}
\text{ and }
Fv =
\begin{bmatrix}
0 \\ 0 \\ 0 \\ 0
\end{bmatrix}.
\end{align*}
Short calculation shows that $\langle \widehat{G}(z),z \rangle > 0$ for all $z = \alpha u + \beta v$ with $\alpha \geq 5$ and $\beta = 1/5$.

We have two goals in the rest of this section. One is to explain what can go wrong, the other is to prepare the notation for \Cref{sec:proof_ell_2,sec:proof_ell_general}.

Let us have a closer look at the scalar product $\langle \widehat{G}(z), z \rangle$. Since the image of $F$ is in $\mathcal{I}^\perp$, there exist functions $F_1, \ldots, F_\ell$ from $H$ to $\mathbb{R}$ such that the vector $F(z)$ (for $z \in H$) takes the form 
\begin{align*}
\begin{bmatrix}
F_1(z)\mathbb{1}_{m_1} \\
\vdots \\
F_\ell(z)\mathbb{1}_{m_\ell} \\
\end{bmatrix} \in \mathbb{R}^{m_1 + \cdots + m_\ell},
\end{align*}
and it is also straightforward to consider the Laplacian matrix $A_\kappa \in \mathbb{R}^{m \times m}$ and any vector $z \in \mathbb{R}^m$ in the block forms
\begin{align*}
A_\kappa = \begin{bmatrix}
A_\kappa(1) &        & 0              \\
            & \ddots &                \\
0           &        & A_\kappa(\ell) \\
\end{bmatrix} \in \mathbb{R}^{(\sum_{i=1}^\ell m_i)\times(\sum_{i=1}^\ell m_i)}
\text{ and }
z = \begin{bmatrix}
z(1) \\
\vdots\\
z(\ell)
\end{bmatrix} \in \mathbb{R}^{\sum_{i=1}^\ell m_i},
\end{align*}
respectively, where the $i$th block corresponds to the $i$th linkage class. (Recall from \Cref{sec:mass_action_systems} that $m_i$ denotes the number of complexes in the $i$th linkage class.) With these, we have
\begin{align} \label{eq:Ghat_z_z_sum}
\langle \widehat{G}(z), z \rangle = e^{F_1(z)} z(1)^\top A_\kappa(1) e^{z(1)} + \cdots + e^{F_\ell(z)} z(\ell)^\top A_\kappa(\ell) e^{z(\ell)}.
\end{align}
We know from \cref{lemma:Gzz_neg} \eqref{it:R0} that for each $i$ the product $z(i)^\top A_\kappa(i) e^{z(i)}$ is negative if $|z(i)|$ is big enough. Further, the function $z(i) \mapsto z(i)^\top A_\kappa(i) e^{z(i)}$ is bounded from above. Still, as the above example shows, the scalar product $\langle \widehat{G}(z), z \rangle$ could be positive no matter how big $|z|$ is. This is because even if $|z|$ is big, there could exist an $i$ such that $|z(i)|$ is small, $z(i)^\top A_\kappa(i) e^{z(i)}$ is positive, and the factors $e^{F_1(z)},\ldots,e^{F_\ell(z)}$ make the positive term dominant among the $\ell$ terms on the r.h.s. of \eqref{eq:Ghat_z_z_sum}.

To overcome the above sketched difficulty, we will truncate the \lq\lq bad parts'' of the ball $\{z\in H ~|~ |z|=r\}$. By \lq\lq bad parts'', we mean points on the boundary of the ball, where the scalar product $\langle \widehat{G}(z), z \rangle$ is not negative, i.e., the vectorfield $\Pi_H \circ \widehat{G} : H \to H$ does not point inwards. After this surgery, we will still have a compact and convex set $\Omega\subseteq\{z\in H ~|~ |z|=r\}$. We will show that the vector field points inwards for all boundary points of $\Omega$.

In \Cref{sec:proof_ell_2,sec:proof_ell_general}, we give the proof of \cref{thm:intersection_for_all_subspaces_and_all_F} for $\ell = 2$ and arbitrary $\ell$, respectively. The case $\ell = 2$ is detailed only for didactic reasons (going directly to the general case is a rather big step in the level of abstraction and therefore the essential features of the approach might remain hidden for the first reading). During the course of the proof of the general case, we will not refer at all to the proof of the case $\ell=2$. Those readers who are short of time, are encouraged to skip \Cref{sec:proof_ell_2}, and proceed directly to \Cref{sec:proof_ell_general}.

\section[alternative section title for TOC]{Proof of \cref{thm:intersection_for_all_subspaces_and_all_F} under $\ell=2$}
\label{sec:proof_ell_2}

Assume throughout this section that $\ell = 2$. Let $\Pi_1:\mathbb{R}^{m_1+m_2} \to \mathbb{R}^{m_1+m_2}$ and $\Pi_2:\mathbb{R}^{m_1+m_2} \to \mathbb{R}^{m_1+m_2}$ be the orthogonal projections defined by
\begin{align*}
\Pi_1z =
\begin{bmatrix}
z(1) \\
0 \\
\end{bmatrix}
\text{ and } 
\Pi_2z =
\begin{bmatrix}
0 \\
z(2) \\
\end{bmatrix}
\text{ for } z \in \mathbb{R}^{m_1+m_2},
\end{align*}
respectively. Further, let
\begin{align}
\label{eq:H1_H2_V}
\begin{split}
H_1 &= H \cap \ker \Pi_2, \\
H_2 &= H \cap \ker \Pi_1, \text{ and } \\
V   &= H \cap (H_1+H_2)^\perp.
\end{split}
\end{align}
Then, by construction, $H$ is the orthogonal direct sum of $H_1$, $H_2$, and $V$. Thus, each $z \in H$ can be written uniquely as
\begin{align} \label{eq:z_z1_z2_y}
z = z_1 + z_2 + y \text{ with } z_1 \in H_1, z_2 \in H_2, \text{ and } y \in V.
\end{align}
The introduced notations will be used throughout this section.

Some useful properties of these objects are summarised in the following lemma.
\begin{lemma}\label{lemma:H1_H2_V}
For any subspace $H\leq \mathbb{R}^{m_1+m_2}$, let $H_1$, $H_2$, and $V$ be as in \eqref{eq:H1_H2_V}. Further, for a $z \in H$, let $z_1$, $z_2$, and $y$ be as in \eqref{eq:z_z1_z2_y}. Then the following two statements hold.
\begin{enumerate}[(i)]
\item\label{it:ortog_subvectors} For each $z\in H$,
\begin{align*}
&z(1) \text{ is the orthogonal sum of } z_1(1) \text{ and } y(1) \text{ and } \\
&z(2) \text{ is the orthogonal sum of } z_2(2) \text{ and } y(2).
\end{align*}
\item\label{it:eps_y} There exists an $\varepsilon > 0$ such that $|y(1)|\geq\varepsilon|y|$ and $|y(2)|\geq\varepsilon|y|$ hold for all $y \in V$.
\end{enumerate}
\end{lemma}
\begin{proof}
First, we prove statement \eqref{it:ortog_subvectors}. Since $z_2 \in \ker \Pi_1$, we have $z_2(1) = 0$, and therefore, $z(1) = z_1(1) + y(1)$. Furthermore, since $z_1 \perp y$ and $z_1(2) = 0$, we have $z_1(1) \perp y(1)$. The case of $z(2)$ is symmetric to the case of $z(1)$.

It is left to show \eqref{it:eps_y}. Since both $\Pi_1|_V$ and $\Pi_2|_V$ are injective (as can be readily seen from the definition of $V$), the composition $(\Pi_2|_V) \circ (\Pi_1|_V)^{-1}$ is a linear bijection between $\ran(\Pi_1|_V)$ and $\ran(\Pi_2|_V)$. Thus, there exists a $c > 0$ such that
\begin{align*}
c |\Pi_1y| \leq |\Pi_2y| \text{ and } c |\Pi_2y| \leq |\Pi_1y| \text{ for all } y \in V.
\end{align*}
Since $|\Pi_1y|=|y(1)|$, $|\Pi_2y|=|y(2)|$, and $|y|^2 = |y(1)|^2 + |y(2)|^2$, there exists an $\varepsilon>0$ such that $|y(1)| \geq \varepsilon |y|$ and $|y(2)| \geq \varepsilon |y|$ for all $y \in V$ (e.g. $\varepsilon = \frac{c}{\sqrt{1+c^2}}$).
\end{proof}

For a triple $(r_0,r_1,r_2)$ with $0 = r_0 < r_1 < r_2$, let
\begin{align} \label{eq:Omega_ell_2}
\Omega = \left\{z \in H ~\Big|~ |z|\leq\sqrt{r_2^2-r_0^2},|z_1|\leq\sqrt{r_2^2-r_1^2},|z_2|\leq\sqrt{r_2^2-r_1^2}\right\}.
\end{align}
Clearly, $\Omega$ is a compact and convex set. Once we prove the following lemma, it also concludes the proof of \cref{thm:intersection_for_all_subspaces_and_all_F} for the special case $\ell = 2$.

\begin{lemma} \label{lemma:Omega_is_good_ell_2}
Let $(\mathcal{X}, \mathcal{C}, \mathcal{R}, \kappa)$ be a weakly reversible mass-action system with $\ell=2$. Let $H$ be an arbitrary subspace of $\mathcal{I}$ and $F:H\to \mathcal{I}^\perp$ be an arbitrary function. Further, let the function $\widehat{G}:H\to\mathbb{R}^m$ be defined by $\widehat{G} = G \circ (\operatorname{Id}+F)$. Then there exists a triple $(r_0,r_1,r_2)$ with $0 = r_0 < r_1 < r_2$ such that the vector field $\Pi_H\circ\widehat{G}:H\to H$ points inwards everywhere on $\partial\Omega$, where $\Omega$ is defined by \eqref{eq:Omega_ell_2}.
\end{lemma}
\begin{proof}
Clearly, $\partial \Omega$ is covered by
\begin{align*}
\left\{z \in \Omega ~\Big|~ |z|=\sqrt{r_2^2-r_0^2}\right\} \bigcup \left(\bigcup_{i=1}^2\left\{z \in \Omega ~\Big|~ |z_i|=\sqrt{r_2^2-r_1^2}\right\}\right).
\end{align*}
The numbers $r_1$ and $r_2$ will be chosen below such that $\langle \widehat{G}(z),n(z) \rangle$ is negative for all $z \in \partial \Omega$, where $n(z)$ denotes any outer normal vector of $\Omega$ at $z$ (the outer normal vector is not necessarily unique). Thus, we have to show that (after choosing $r_1$ and $r_2$ appropriately)
\begin{align}
\label{eq:neg_on_boundary_z}
\langle \widehat{G}(z), z \rangle &< 0 \text{ for all } z \in \Omega \text{ with } |z|=\sqrt{r_2^2-r_0^2}, \\
\label{eq:neg_on_boundary_z1}
\langle \widehat{G}(z), z_1 \rangle &< 0 \text{ for all } z \in \Omega \text{ with } |z_1|=\sqrt{r_2^2-r_1^2}, \\
\label{eq:neg_on_boundary_z2}
\langle \widehat{G}(z), z_2 \rangle &< 0 \text{ for all } z \in \Omega \text{ with } |z_2|=\sqrt{r_2^2-r_1^2}.
\end{align}

First, we will set $r_1$ such that \eqref{eq:neg_on_boundary_z} holds. Since
\begin{align}\label{eq:Ghat_z_z_ell_2}
\langle \widehat{G}(z), z \rangle = e^{F_1(z)} z(1)^\top A_\kappa(1) e^{z(1)} + e^{F_2(z)} z(2)^\top A_\kappa(2) e^{z(2)},
\end{align}
it suffices to show that both $|z(1)|$ and $|z(2)|$ are at least $R=\max(R(1),R(2))$, where $R(i)$ is the threshold guaranteed to exist by \cref{lemma:Gzz_neg} \eqref{it:R0} when applied to the $i$th linkage class ($i = 1,2$). (Instead of referring to \cref{lemma:Gzz_neg} \eqref{it:R0}, one could equivalently refer to \cref{cor:Gzz_neg_error_rho} with $\varrho=0$.)

Assuming $|z| = \sqrt{r_2^2-r_0^2}$, and using also $|z_2| \leq \sqrt{r_2^2-r_1^2}$, we have
\begin{align*}
|z_1|^2 + |y|^2 = |z|^2 - |z_2|^2 \geq (r_2^2-r_0^2) - (r_2^2-r_1^2) = r_1^2-r_0^2.
\end{align*}
Therefore, 
\begin{align*}
|z(1)|^2 &= |z_1(1)|^2+|y(1)|^2 = |z_1|^2+|y(1)|^2 \geq \\
&\geq
\begin{cases}
|z_1|^2 \geq \frac{r_1^2-r_0^2}{2}, & \text{ if } |y|^2\leq\frac{r_1^2-r_0^2}{2}, \\
|y(1)|^2 \geq \varepsilon^2 |y|^2 \geq \varepsilon^2 \frac{r_1^2-r_0^2}{2}, & \text{ if } |y|^2\geq\frac{r_1^2-r_0^2}{2}, \\
\end{cases}
\end{align*}
where the first equality comes from \cref{lemma:H1_H2_V} \eqref{it:ortog_subvectors}, while $\varepsilon$ is as in \cref{lemma:H1_H2_V} \eqref{it:eps_y}. Thus, in any case, $|z(1)| \geq \varepsilon\sqrt{\frac{r_1^2-r_0^2}{2}}$. Analogous reasoning shows that $|z(2)| \geq \varepsilon\sqrt{\frac{r_1^2-r_0^2}{2}}$. Let us fix $r_1>r_0$ such that $\varepsilon\sqrt{\frac{r_1^2-r_0^2}{2}}>R$. Then both terms on the r.h.s. of \eqref{eq:Ghat_z_z_ell_2} are negative and we have proven \eqref{eq:neg_on_boundary_z}.

Finally, we will set $r_2$ such that both \eqref{eq:neg_on_boundary_z1} and \eqref{eq:neg_on_boundary_z2} hold. Fix $i \in \{1,2\}$. Since
\begin{align*}
z_1 = \begin{bmatrix}
z_1(1) \\
0
\end{bmatrix} \text{ and }
z_2 = \begin{bmatrix}
0 \\
z_2(2)
\end{bmatrix},
\end{align*}
we have
\begin{align}\label{eq:Ghat_z_z1_ell_2}
\langle \widehat{G}(z),z_i \rangle = e^{F_i(z)} z_i(i) A_\kappa(i) e^{z_i(i)+y(i)}.
\end{align}
Assuming $|z_i| = \sqrt{r_2^2-r_1^2}$, and using also $|z| \leq \sqrt{r_2^2-r_0^2}$, we have
\begin{align*}
|y(i)|^2 \leq |y|^2 = |z|^2 - |z_1|^2 - |z_2|^2 \leq (r_2^2-r_0^2)-(r_2^2-r_1^2) = r_1^2.
\end{align*}
Let now $r_2$ be such that $\sqrt{r_2^2-r_1^2} > \max(R_{r_1}(1),R_{r_1}(2))$, where $R_{r_1}(i)$ is the threshold guaranteed to exist by \cref{cor:Gzz_neg_error_rho} when applied to the $i$th linkage class with $\varrho = r_1$. Then, by \cref{cor:Gzz_neg_error_rho}, the r.h.s. of \eqref{eq:Ghat_z_z1_ell_2} is negative. This concludes the proof of \eqref{eq:neg_on_boundary_z1} and \eqref{eq:neg_on_boundary_z2}, also the proof of the lemma, and, in turn, also the proof of \cref{thm:intersection_for_all_subspaces_and_all_F} in the special case $\ell=2$.
\end{proof}

\section[alternative section title for TOC]{Proof of \cref{thm:intersection_for_all_subspaces_and_all_F} for arbitrary $\ell$}
\label{sec:proof_ell_general}

As preparation for proving \cref{thm:intersection_for_all_subspaces_and_all_F} in general (i.e., without any restriction on the number of linkage classes), we first collect some general observations.

Fix $Q,Q' \subseteq \{1,\ldots,\ell\}$ with $Q'\subseteq Q$ and a subspace $H \leq \mathbb{R}^{\sum_{i=1}^\ell m_i}$ for the part before \cref{lemma:H1_H2_V_general}. Further, let $Q^c = \{1,\ldots,\ell\}\setminus Q$.

Let $\Pi_Q : \mathbb{R}^{\sum_{i=1}^\ell m_i} \to \mathbb{R}^{\sum_{i=1}^\ell m_i}$ be the orthogonal projection defined by
\begin{align*}
(\Pi_Qz)(i) = \begin{cases} z(i), & \text{ if } i\in Q, \\
                            0,    & \text{ if } i\notin Q, \end{cases}
\end{align*}
or equivalently, with vectorial notation, $(\Pi_Qz)(Q) = z(Q)$ and $(\Pi_Qz)(Q^c) = 0$.

Let us define the subspace $H_Q$ of $H$ by
\begin{align*}
H_Q = H \cap \ker \Pi_{Q^c}.
\end{align*}
I.e., $H_Q$ consists of those elements of $H$ whose support is contained in the blocks corresponding to $Q$. For $z \in H$, we will use the shorthand notation $z_Q$ for $\Pi_{H_Q}z$ (i.e., $z_Q$ is the component of $z$ lying in $H_Q$).

Then the subspace $H_Q$ is the orthogonal direct sum of $H_{Q'}$, $H_{Q \setminus Q'}$, and $V_{Q,Q'}$, where $V_{Q,Q'}$ is by definition $H_Q \cap (H_{Q'} + H_{Q \setminus Q'})^\perp$. Thus, for each $z \in H$, we have
\begin{align} \label{eq:z_Q_ortog_decomp}
z_Q = z_{Q'} + z_{Q \setminus Q'} + y,
\end{align}
where $y$ is by definition the component of $z_Q$ lying in $V_{Q,Q'}$. (Our notations are already cumbersome enough, so we do not indicate in the notation of $y$ its dependence on $Q$, $Q'$, and $z$. It will not cause any misunderstanding.) The three components of $z_Q$ on the r.h.s. of \eqref{eq:z_Q_ortog_decomp} are pairwise orthogonal. Further, note that the supports of the vectors $z_Q$, $z_{Q'}$, $z_{Q \setminus Q'}$, and $y$ lie in the blocks corresponding to $Q$.

Some useful properties of the introduced objects are summarised in the following lemma.
\begin{lemma}\label{lemma:H1_H2_V_general}
Fix a subspace $H\leq \mathbb{R}^{\sum_{i=1}^\ell m_i}$ and a pair $(Q,Q')$ with $Q' \subseteq Q \subseteq \{1,\ldots,\ell\}$. With the introduced notations, the following two statements hold.
\begin{enumerate}[(i)]
\item\label{it:ortog_subvectors_general} For each $z\in H$, the vector $z_Q(Q')$ is the orthogonal sum of $z_{Q'}(Q')$ and $y(Q')$, where $y$ is as in \eqref{eq:z_Q_ortog_decomp}.
\item\label{it:eps_y_general} There exists an $\varepsilon > 0$ such that $|y(Q')|\geq\varepsilon|y|$ holds for all $y \in V_{Q,Q'}$.
\end{enumerate}
\end{lemma}
\begin{proof}
First, we prove statement \eqref{it:ortog_subvectors_general}. Since $z_{Q\setminus Q'} \in \ker \Pi_{Q'}$, we have $z_{Q\setminus Q'}(Q') = 0$, and therefore, $z_Q(Q') = z_{Q'}(Q') + y(Q')$. Furthermore, since $z_{Q'} \perp y$ and $z_{Q'}(Q\setminus Q') = 0$, we have $z_{Q'}(Q') \perp y(Q')$.

It is left to show \eqref{it:eps_y_general}. In order to ease the notation, we write $V$ for $V_{Q,Q'}$. Since both $\Pi_{Q'}|_V$ and $\Pi_{Q \setminus Q'}|_V$ are injective (as can be readily seen from the definition of $V$), the composition $(\Pi_{Q \setminus Q'}|_V) \circ (\Pi_{Q'}|_V)^{-1}$ is a linear bijection between $\ran (\Pi_{Q'}|_V)$ and $\ran (\Pi_{Q \setminus Q'}|_V)$. Thus, there exists a $c > 0$ such that
\begin{align*}
c |\Pi_{Q'}y| \leq |\Pi_{Q \setminus Q'}y| \text{ and } c |\Pi_{Q \setminus Q'}y| \leq |\Pi_{Q'}y| \text{ for all } y \in V.
\end{align*}
Since $|\Pi_{Q'}y|=|y(Q')|$, $|\Pi_{Q\setminus Q'}y|=|y(Q\setminus Q')|$, and $|y|^2 = |y(Q')|^2 + |y(Q\setminus Q')|^2$, there exists an $\varepsilon>0$ such that $|y(Q')| \geq \varepsilon |y|$ holds for all $y \in V$ (e.g. $\varepsilon = \frac{c}{\sqrt{1+c^2}}$).
\end{proof}

For a tuple $(r_0,r_1,\ldots,r_\ell)$ with $0 = r_0 < r_1 < \cdots < r_\ell$, let
\begin{align} \label{eq:Omega_general}
\Omega = \left\{z \in H ~\Big|~ |z_Q|\leq \sqrt{r_\ell^2-r_{\ell-|Q|}^2}\text{ for all }\emptyset\neq Q\subseteq\{1,\ldots,\ell\}\right\}.
\end{align}
Clearly, $\Omega$ is a compact and convex set. Once we prove the following lemma, it also concludes the proof of \cref{thm:intersection_for_all_subspaces_and_all_F} for arbitrary $\ell$.

\begin{lemma} \label{lemma:Omega_is_good_ell_general}
Let $(\mathcal{X}, \mathcal{C}, \mathcal{R}, \kappa)$ be a weakly reversible mass-action system. Let $H$ be an arbitrary subspace of $\mathcal{I}$ and $F:H\to \mathcal{I}^\perp$ be an arbitrary function. Further, let the function $\widehat{G}:H\to\mathbb{R}^m$ be defined by $\widehat{G} = G \circ (\operatorname{Id}+F)$. Then there exists a tuple $(r_0,r_1,\ldots,r_\ell)$ with $0 = r_0 < r_1 < \cdots < r_\ell$ such that the vector field $\Pi_H\circ\widehat{G}:H\to H$ points inwards everywhere on $\partial\Omega$, where $\Omega$ is defined by \eqref{eq:Omega_general}.
\end{lemma}
\begin{proof}
Clearly, $\partial \Omega$ is covered by
\begin{align*}
\bigcup_{\emptyset\neq Q\subseteq\{1,\ldots,\ell\}}
\left\{z \in \Omega ~\Big|~ |z_Q|=\sqrt{r_\ell^2-r_{\ell-|Q|}^2}\right\}.
\end{align*}
We will set the values of $r_1,\ldots,r_\ell$ such that for all $\emptyset\neq Q\subseteq\{1,\ldots,\ell\}$ we have
\begin{align} \label{eq:scalar_product_with_Q}
\langle \widehat{G}(z),z_Q \rangle < 0 \text{ for all } z \in \Omega \text{ with } |z_Q| = \sqrt{r_\ell^2-r_{\ell-|Q|}^2}.
\end{align}
The way we proceed is the following. First we set the value of $r_1$ such that \eqref{eq:scalar_product_with_Q} holds for $Q = \{1,\ldots,\ell\}$. Then we set the value of $r_2$ such that \eqref{eq:scalar_product_with_Q} holds for all $Q \subseteq \{1,\ldots,\ell\}$ with $|Q|=\ell-1$. And so on. The final step is to set the value of $r_\ell$ such that \eqref{eq:scalar_product_with_Q} holds for all $Q \subseteq \{1,\ldots,\ell\}$ with $|Q|=1$.

Fix $k \in \{1,\ldots,\ell\}$ and assume that the values of $r_0,r_1,\ldots,r_{k-1}$ are already fixed such that \eqref{eq:scalar_product_with_Q} holds for all $Q \subseteq \{1,\ldots,\ell\}$ with $|Q|\geq \ell-(k-2)$. Our goal is to set the value of $r_k$ such that \eqref{eq:scalar_product_with_Q} holds for all $Q \subseteq \{1,\ldots,\ell\}$ with $|Q| = \ell-(k-1)$.

Fix now $Q \subseteq \{1,\ldots,\ell\}$ with $|Q| = \ell-(k-1)$. Since $z_Q(i)=0$ for all $z \in H$ and for all $i \in Q^c$, we have
\begin{align*}
\langle \widehat{G}(z),z_Q \rangle = \sum_{i\in Q}e^{F_i(z)}z_Q(i)A_\kappa(i)e^{z(i)}.
\end{align*}
Since for all $z \in H$ we have the orthogonal decomposition $z = z_Q + z_{Q^c}+ y$, we also have $z(i) = z_Q(i) + y(i)$ for all $i \in Q$. Therefore,
\begin{align} \label{eq:Ghat_z_z_Q_as_a_sum}
\langle \widehat{G}(z),z_Q \rangle = \sum_{i\in Q}e^{F_i(z)}z_Q(i)A_\kappa(i)e^{z_Q(i)+y(i)}.
\end{align}
For all $z \in \Omega$ with $|z_Q| = \sqrt{r_\ell^2-r_{k-1}^2}$ and for all $i \in Q$, using also $|z| \leq \sqrt{r_\ell^2-r_0^2}$, we have
\begin{align} \label{eq:error_upper_bound}
|y(i)|^2 \leq |y|^2 = |z|^2 - |z_Q|^2 - |z_{Q^c}|^2 \leq
(r_\ell^2 - r_0^2) - (r_\ell^2 - r^2_{k-1}) = r_{k-1}^2.
\end{align}
Let now $R = \max_{i\in Q}R_{r_{k-1}}(i)$, where $R_{r_{k-1}}(i)$ is the threshold guaranteed to exist by \cref{cor:Gzz_neg_error_rho} when applied to the $i$th linkage class with $\varrho = r_{k-1}$. Thus, taking also into account \eqref{eq:error_upper_bound}, if we can set $r_k$ such that $|z_Q(i)| \geq R$ for all $i \in Q$ then all the terms on the r.h.s. of \eqref{eq:Ghat_z_z_Q_as_a_sum} are negative.

Fix $i \in Q$ and a $z \in \Omega$ with $|z_Q| = \sqrt{r_\ell^2-r_{k-1}^2}$. Then with $Q'=\{i\}$ we have $z_Q = z_{Q'} + z_{Q \setminus Q'} + y$, where the three components of $z_Q$ on the r.h.s. are pairwise orthogonal. Thus, using also $|z_{Q\setminus Q'}| \leq \sqrt{r_\ell^2-r_k^2}$ (since $|Q\setminus Q'|=\ell-k$), we have
\begin{align*}
|z_{Q'}|^2 + |y|^2 = |z_Q|^2 - |z_{Q\setminus Q'}|^2 \geq (r_\ell^2-r_{k-1}^2) - (r_\ell^2-r_k^2) = r_k^2-r_{k-1}^2.
\end{align*}
Therefore, 
\begin{align*}
|z_Q(i)|^2 &= |z_{Q'}(i)|^2+|y(i)|^2 = |z_{Q'}|^2+|y(i)|^2 \geq \\
&\geq
\begin{cases}
|z_{Q'}|^2 \geq \frac{r_k^2-r_{k-1}^2}{2}, & \text{ if } |y|^2\leq\frac{r_k^2-r_{k-1}^2}{2}, \\
|y(i)|^2 \geq \varepsilon(Q,i)^2 |y|^2 \geq \varepsilon(Q,i)^2 \frac{r_k^2-r_{k-1}^2}{2}, & \text{ if } |y|^2\geq\frac{r_k^2-r_{k-1}^2}{2}, \\
\end{cases}
\end{align*}
where the first equality comes from \cref{lemma:H1_H2_V_general} \eqref{it:ortog_subvectors_general}, while $\varepsilon(Q,i)$ is as in \cref{lemma:H1_H2_V_general} \eqref{it:eps_y_general} when applied with $Q' = \{i\}$. Thus, in any case, $|z_Q(i)| \geq \varepsilon_k\sqrt{\frac{r_k^2-r_{k-1}^2}{2}}$, where $\varepsilon_k$ is the minimum of the numbers $\varepsilon(Q,j)$ over all pairs $(Q,j)$ with $|Q|=\ell-(k-1)$ and $j \in Q$.

Choose $r_k$ so big that $\varepsilon_k\sqrt{\frac{r_k^2-r_{k-1}^2}{2}} \geq R$ holds. This concludes the proof this lemma, the proof of \cref{thm:intersection_for_all_subspaces_and_all_F}, and, in turn, the proof of our main result, \cref{thm:existence_in_each_class}.
\end{proof}

\appendix

\section{A tool that is crucial for proving \cref{lemma:use_birch_in_proof}}
\label{sec:app_birch}

The following lemma is the heart of the proof of \cref{lemma:use_birch_in_proof}.

\begin{lemma}\label{lemma:psi}
Let $Y \in \mathbb{R}^{n \times m}$, $I \in \mathbb{R}^{m \times d}$, $\mathcal{I} = \ran I$, $K = \Pi_\mathcal{I}(\ran Y^\top)$, $p \in \mathbb{R}^n_+$, and $\mathcal{P} = (p+\ran(YI))\cap \mathbb{R}^n_+$. Then the map $\Psi:\mathcal{P} \to K$, defined by $\Psi(x) = \Pi_{\mathcal{I}}Y^\top\log(x)$, is a bijection between $\mathcal{P}$ and $K$.
\end{lemma}

\cref{lemma:psi} is a consequence of \cref{lemma:birch} below. The latter, due to its relation to Martin Birch's work \cite{birch:1963}, is sometimes called Birch's Theorem, see e.g. \cite[Theorem 1.10]{pachter:sturmfels:2005} and \cite[Theorem 5.12]{gopalkrishnan:miller:shiu:2014}. Several versions/generalizations appear in the CRNT literature, see e.g. \cite[Section 4]{horn:jackson:1972}, \cite[Corollary 4.14]{feinberg:1979}, \cite[Appendix B]{feinberg:1995a}, \cite[Lemma IV.1]{sontag:2001}, \cite[Section 6]{gunawardena:2003}, \cite[Propositon 10]{craciun:dickenstein:shiu:sturmfels:2009}, \cite[Section 5]{gopalkrishnan:miller:shiu:2014}, \cite[Theorem 13]{mueller:hofbauer:regensburger:2018}, and \cite[Theorem 6.7]{craciun:mueller:pantea:yu:2018}.

\begin{lemma}\label{lemma:birch}
Let $\mathcal{S} \leq \mathbb{R}^n_+$ and $p \in \mathbb{R}^n_+$. Then for each $x^* \in \mathbb{R}^n_+$ there exists a unique $x \in \mathbb{R}^n_+$ such that $x-p \in \mathcal{S}$ and $\log x - \log x^* \in \mathcal{S}^\perp$.
\end{lemma}

We explain in the rest of this section how \cref{lemma:psi} follows from \cref{lemma:birch}. Since the elements of $K$ are exactly those that can be written in the form $\Pi_\mathcal{I}Y^\top \log x^*$ for some $x^* \in \mathbb{R}^n_+$, the bijectivity of $\Psi$ is equivalent to the following statement. For each $x^* \in \mathbb{R}^n_+$ there exists a unique $x \in \mathcal{P}$ such that
\begin{align} \label{eq:birch}
\Pi_\mathcal{I}Y^\top \log x = \Pi_\mathcal{I}Y^\top \log x^*.
\end{align}
Since $\mathcal{I} = \ran I$, the equality \eqref{eq:birch} is equivalent to
\begin{align*}
I^\top Y^\top (\log x - \log x^*) = 0.
\end{align*}
Therefore, the bijectivity of $\Psi$ can be rephrased in the following way. For each $x^* \in \mathbb{R}^n_+$ there exists a unique $x \in \mathbb{R}^n_+$ such that $x-p \in \ran(YI)$ and $\log x - \log x^* \in \ran(YI)^\perp$. Finally, the latter holds by \cref{lemma:birch} with $\mathcal{S} = \ran(YI)$.

\section{Proof of \cref{lemma:yi_exp_minusyiminus1}}
\label{sec:proof_yi_exp_minusyiminus1}

The case $M = 1$ is trivial ($y_1\to\infty$ as $y_1\to\infty$).

Let now $M\geq1$ and assume that the lemma is true for $M$. We will show that it is also true for $M+1$. For fixed $y_0=0$, $y_1, \ldots, y_{M-1}\geq 0$, and $y_{M+1}>0$, the function
\begin{align*}
\mathbb{R} \ni y_M \stackrel{\Phi}{\mapsto} \sum_{i=1}^{M+1} e^{-y_{i-1}} y_i \in \mathbb{R}
\end{align*}
is monotone decreasing on the interval $(-\infty,y_{M-1}+\log y_{M+1}]$, while monotone increasing on the interval $[y_{M-1}+\log y_{M+1},\infty)$. Thus, the restriction $\Phi|_{[0,\infty)}$ attains its minimum either at $0$ or at $y_{M-1}+\log y_{M+1}$.

Note that
\begin{align*}
&\Phi(y_{M-1}+\log y_{M+1}) = \\
&= \left(\sum_{i=1}^{M-1} e^{-y_{i-1}}y_i\right) + e^{-y_{M-1}}(y_{M-1}+\log y_{M+1}) + e^{-y_{M-1}-\log y_{M+1}}y_{M+1}\geq \\
&\geq \left(\sum_{i=1}^{M-1} e^{-y_{i-1}}y_i\right) + e^{-y_{M-1}} \log y_{M+1}\to \infty \text{ as } \log y_{M+1} \to \infty
\end{align*}
due to the induction hypothesis. Further,
\begin{align*}
\Phi(0) = \left(\sum_{i=1}^{M-1} e^{-y_{i-1}}y_i\right) + y_{M+1} \geq y_{M+1} \to \infty \text{ as } y_{M+1} \to \infty.
\end{align*}
This concludes the proof.

\section{The acyclic digraph of the implications}
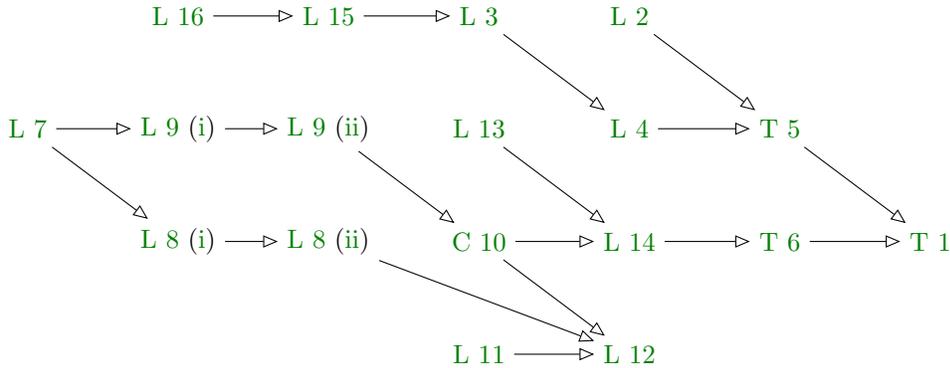
\label{sec:acyclic_digraph}

\begin{center}
\begin{tikzpicture}[auto]

\node (A) at ( 0, 0  ) {\Cref{lemma:yi_exp_minusyiminus1}};
\node (B) at ( 2, 0  ) {\Cref{lemma:Gzz_neg_many_kappas} \eqref{it:L0}};
\node (C) at ( 4, 0  ) {\Cref{lemma:Gzz_neg_many_kappas} \eqref{it:R0}};
\node (D) at ( 6,-1.5) {\Cref{cor:Gzz_neg_error_rho}};
\node (E) at ( 8,-1.5) {\Cref{lemma:Omega_is_good_ell_general}};
\node (F) at (10,-1.5) {\Cref{thm:intersection_for_all_subspaces_and_all_F}};
\node (G) at (10, 0  ) {\Cref{thm:Epl_cap_P_hatGK_cap_Kperp}};
\node (H) at (12,-1.5) {\Cref{thm:existence_in_each_class}};
\node (I) at ( 2,-1.5) {\Cref{lemma:Gzz_neg} \eqref{it:L0_many_kappas}};
\node (J) at ( 4,-1.5) {\Cref{lemma:Gzz_neg} \eqref{it:R0_many_kappas}};
\node (K) at ( 6,-3  ) {\Cref{lemma:H1_H2_V}};
\node (L) at ( 8,-3  ) {\Cref{lemma:Omega_is_good_ell_2}};
\node (M) at ( 6, 0  ) {\Cref{lemma:H1_H2_V_general}};
\node (N) at ( 8, 1.5) {\Cref{lemma:GYTlogP_ranYTperp}};
\node (O) at ( 6, 1.5) {\Cref{lemma:use_birch_in_proof}};
\node (P) at ( 8, 0)   {\Cref{lemma:hatGK_cap_Kperp}};
\node (Q) at ( 4, 1.5) {\Cref{lemma:psi}};
\node (R) at ( 2, 1.5) {\Cref{lemma:birch}};

\draw[-{Triangle[open,angle=45:6pt]}] (A) -- (B);
\draw[-{Triangle[open,angle=45:6pt]}] (B) -- (C);
\draw[-{Triangle[open,angle=45:6pt]}] (C) -- (D);
\draw[-{Triangle[open,angle=45:6pt]}] (D) -- (E);
\draw[-{Triangle[open,angle=45:6pt]}] (E) -- (F);
\draw[-{Triangle[open,angle=45:6pt]}] (F) -- (H);
\draw[-{Triangle[open,angle=45:6pt]}] (G) -- (H);

\draw[-{Triangle[open,angle=45:6pt]}] (A) -- (I);
\draw[-{Triangle[open,angle=45:6pt]}] (I) -- (J);

\draw[-{Triangle[open,angle=45:6pt]}] (J) -- (L);
\draw[-{Triangle[open,angle=45:6pt]}] (D) -- (L);
\draw[-{Triangle[open,angle=45:6pt]}] (K) -- (L);

\draw[-{Triangle[open,angle=45:6pt]}] (M) -- (E);

\draw[-{Triangle[open,angle=45:6pt]}] (N) -- (G);
\draw[-{Triangle[open,angle=45:6pt]}] (O) -- (P);
\draw[-{Triangle[open,angle=45:6pt]}] (P) -- (G);
\draw[-{Triangle[open,angle=45:6pt]}] (R) -- (Q);
\draw[-{Triangle[open,angle=45:6pt]}] (Q) -- (O);
\end{tikzpicture}
\end{center}

\section*{Acknowledgments}

The author thanks Josef Hofbauer, Stefan M\"uller, and Georg Regensburger for useful discussions, Martin Feinberg for historical details on \cite{deng:feinberg:jones:nachman:2011}, and two anonymous referees for helpful comments.

\bibliographystyle{siamplain}
\bibliography{references}

\def\cprime{$'$}
\begin{thebibliography}{10}

\bibitem{birch:1963}
{\sc M.~W. Birch}, {\em Maximum likelihood in three-way contingency tables}, J.
  Roy. Statist. Soc. Ser. B, 25 (1963), pp.~220--233.

\bibitem{boros:2013a}
{\sc B.~Boros}, {\em On the existence of the positive steady states of weakly
  reversible deficiency-one mass action systems}, Math. Biosci., 245 (2013),
  pp.~157--170.

\bibitem{boros:2013c}
{\sc B.~Boros}, {\em On the {P}ositive {S}teady {S}tates of {D}eficiency-{O}ne
  {M}ass {A}ction {S}ystems}, PhD thesis, E\"otv\"os Lor\'and University,
  Budapest, 2013.

\bibitem{boros:craciun:yu:2018}
{\sc B.~Boros, G.~Craciun, and P.~Y. Yu}, {\em Weakly reversible mass-action
  systems with infinitely many positive steady states}.
\newblock In preparation, 2018.

\bibitem{craciun:dickenstein:shiu:sturmfels:2009}
{\sc G.~Craciun, A.~Dickenstein, A.~Shiu, and B.~Sturmfels}, {\em Toric
  dynamical systems}, J. Symb. Comput., 44 (2009), pp.~1551--1565.

\bibitem{craciun:mueller:pantea:yu:2018}
{\sc G.~Craciun, S.~M{\"u}ller, C.~Pantea, and P.~Y. Yu}, {\em A generalization
  of {B}irch's theorem and vertex-balanced steady states for generalized
  mass-action systems}.
\newblock 2018, \href{http://arxiv.org/abs/1802.06919} {arXiv:1802.06919}.

\bibitem{craciun:nazarov:pantea:2013}
{\sc G.~Craciun, F.~Nazarov, and C.~Pantea}, {\em Persistence and {P}ermanence
  of {M}ass-{A}ction and {P}ower-{L}aw {D}ynamical {S}ystems}, SIAM J. Appl.
  Math., 73 (2013), pp.~305--329.

\bibitem{deng:feinberg:jones:nachman:2011}
{\sc J.~Deng, M.~Feinberg, C.~Jones, and A.~Nachman}, {\em On the steady states
  of weakly reversible chemical reaction networks}.
\newblock 2011, \href{http://arxiv.org/abs/1111.2386} {arXiv:1111.2386}.

\bibitem{feinberg:1979}
{\sc M.~Feinberg}, {\em Lectures on chemical reaction networks}.
\newblock 4.5 out of 9 lectures delivered at the Mathematics Research Center,
  University of Wisconsin, Fall 1979. Available at
  \url{http://crnt.osu.edu/LecturesOnReactionNetworks}.

\bibitem{feinberg:1987}
{\sc M.~Feinberg}, {\em Chemical reaction network structure and the stability
  of complex isothermal reactors - {I. The} {D}eficiency {Z}ero and the
  {D}eficiency {O}ne {T}heorems}, Chemical Engineering Science, 42 (1987),
  pp.~2229--2268.

\bibitem{feinberg:1995a}
{\sc M.~Feinberg}, {\em The existence and uniqueness of steady states for a
  class of chemical reaction networks}, Arch. Rational Mech. Anal., 132 (1995),
  pp.~311--370.

\bibitem{feinberg:horn:1977}
{\sc M.~Feinberg and F.~Horn}, {\em Chemical mechanism structure and the
  coincidence of the stoichiometric and kinetic subspaces}, Arch. Rational
  Mech. Anal., 66 (1977), pp.~83--97.

\bibitem{gopalkrishnan:miller:shiu:2014}
{\sc M.~Gopalkrishnan, E.~Miller, and A.~Shiu}, {\em A geometric approach to
  the global attractor conjecture}, SIAM Journal on Applied Dynamical Systems,
  13 (2014), pp.~758--797.

\bibitem{gunawardena:2003}
{\sc J.~Gunawardena}, {\em Chemical reaction network theory for in-silico
  biologists}, 2003, \url{http://vcp.med.harvard.edu/papers/crnt.pdf}.

\bibitem{horn:jackson:1972}
{\sc F.~Horn and R.~Jackson}, {\em General mass action kinetics}, Arch.
  Rational Mech. Anal., 47 (1972), pp.~81--116.

\bibitem{mueller:hofbauer:regensburger:2018}
{\sc S.~M{\"u}ller, J.~Hofbauer, and G.~Regensburger}, {\em On the bijectivity
  of families of exponential/generalized polynomial maps}.
\newblock 2018, \href{http://arxiv.org/abs/1804.01851} {arXiv:1804.01851}.

\bibitem{orlov:rozonoer:1984b}
{\sc V.~Orlov and L.~Rozono{\`e}r}, {\em The macrodynamics of open systems and
  the variational principle of the local potential. {II}. {A}pplications}, J.
  Franklin Inst., 318 (1984), pp.~315--347.

\bibitem{pachter:sturmfels:2005}
{\sc L.~Pachter and B.~Sturmfels}, eds., {\em {Algebraic statistics for
  computational biology}}, Cambridge University Press, Cambridge, 2005.

\bibitem{simon:1995}
{\sc P.~L. Simon}, {\em Globally attracting domains in two-dimensional
  reversible chemical dynamical systems}, Ann. Univ. Sci. Budapest. Sect.
  Comput., 15 (1995), pp.~179--200.

\bibitem{sontag:2001}
{\sc E.~D. Sontag}, {\em Structure and stability of certain chemical networks
  and applications to the kinetic proofreading model of {T}-cell receptor
  signal transduction}, IEEE Trans. Automat. Control, 46 (2001),
  pp.~1028--1047.

\end{thebibliography}
\end{document}